\documentclass[10pt,reqno]{amsart}
\usepackage{graphicx, amssymb, color,pdfsync}
\usepackage[all,cmtip]{xy}
\numberwithin{equation}{section}
\usepackage{mathrsfs,soul,ulem}
\usepackage{tikz}
\usetikzlibrary{matrix,arrows}

\usepackage{hyperref}
\hypersetup{
    colorlinks=true,       
    linkcolor=blue,          
    citecolor=blue,        
    filecolor=blue,      
    urlcolor=blue           
}

\newtheorem{theo}{Theorem}[section]

\newtheorem{lemm}[theo]{Lemma}
\newtheorem{coro}[theo]{Corollary}
\newtheorem{prop}[theo]{Proposition}

\theoremstyle{remark}
\newtheorem{rema}[theo]{\bf Remark}

\newtheorem{example}{\bf Example}

\begin{document}

\title[On the fiber product of Riemann surfaces]{On the fiber product of Riemann surfaces}
\author{Rub\'en A. Hidalgo}
\author{Sebasti\'an Reyes-Carocca}
\author{Ang\'elica Vega}
\thanks{This work was partially supported by Project FONDECYT 1150003, Project Postdoctoral
FONDECYT 3160002, Project Anillo ACT 1415 PIA-CONICYT and Project Redes Etapa Inicial 2017-170071}


\address{Departamento de Matem\'atica y Estad\'{\i}stica, Universidad de La Frontera. Avenida Francisco Salazar 01145, Temuco, Chile}
\email{ruben.hidalgo@ufrontera.cl, sebastian.reyes@ufrontera.cl}
\address{Escuela Colombiana de Ingenier\'ia, Avenida Carrera 45 305-59, Bogot\'a, Colombia}
\email{amvegam@gmail.com}

\begin{abstract}
Let $S_{0}, S_{1}$ and $S_{2}$ be connected Riemann surfaces and let $\beta_{1}:S_{1} \to S_{0}$ and $\beta_{2}:S_{2} \to S_{0}$ be surjective holomorphic maps. The associated fiber product
 $S_{1} \times_{(\beta_{1},\beta_{2})} S_{2}$  has the structure of a  singular Riemann surface, endowed with a canonical map $\beta$ to $S_{0}$ satisfying that $\beta_{j} \circ \pi_{j}=\beta$, where $\pi_{j}$ is coordinate projection onto $S_{j}$. In this paper we provide a Fuchsian description of the fiber product and obtain that 
 if one the maps $\beta_{j}$ is a regular branched cover, then all its irreducible components are isomorphic. 
 In the case that both $\beta_{j}$ are of finite degree,  we observe that the number of irreducible components is bounded above by the greatest common divisor of the two degrees; we study the irreducibility of the fiber product.
 In the case that $S_{0}=\widehat{\mathbb C}$, and $S_{1}$ and $S_{2}$ are compact, we define the strong field of moduli of the pair $(S_{1} \times_{(\beta_{1},\beta_{2})} S_{2},\beta)$ and observe that this field coincides with the minimal field containing the fields of moduli of both pairs $(S_{1},\beta_{1})$ and $(S_{2},\beta_{2})$. Finally, in the case that the fiber product is a connected Riemann surface, we provide an isogenous decomposition of its Jacobian variety.
\end{abstract}

\maketitle

\section{Introduction}

In the category of topological spaces the fiber product is the unique solution of a certain universal problem (see Section \ref{Sec:productofibrado}). By restricting to some subcategory, it might be that the fiber product of two of its elements does not belong  to it. Subcategories for which the fiber product still insides it are the one of (not necessarily irreducible nor smooth) algebraic varieties and the one of schemes (see, for example, \cite{Iitaka}). By contrast,  subcategories over which the fiber product may not still be in there are, for instance, the one of smooth algebraic varieties as in general this process produces singularities and the one of irreducible algebraic varieties as this process may produce reducible objects; in particular, the subcategory of Riemann surfaces. 
The fiber product has been a main tool for constructing examples and counterexamples in algebraic geometry.

This article is devoted to study the fiber product at the level of connected Riemann surfaces.
Let $S_{0}$, $S_{1}$ and $S_{2}$ be connected (and not necessarily compact) Riemann surfaces, let $\beta_{1}:S_{1} \to S_{0}$ and $\beta_{2}:S_{2} \to S_{0}$ be surjective holomorphic maps, and let $S_{1} \times_{(\beta_{1},\beta_{2})} S_{2}$ be its associated fiber product.
There is a natural map \begin{equation} \label{lente}\beta:S_{1} \times_{(\beta_{1},\beta_{2})} S_{2} \to S_{0}\end{equation} such that $\beta=\beta_{j} \circ \pi_{j}$, where $\pi_{j}:S_{1} \times_{(\beta_{1},\beta_{2})} S_{2} \to S_{j}$ is the natural projection.

The fiber product $S_{1} \times_{(\beta_{1},\beta_{2})} S_{2}$ is a (possible non-connected) one-dimensional complex analytic closed subspace of the complex surface $S_{1} \times S_{2}$, which might or might not be smooth. 

The singular locus ${\rm Sing} \subset S_{1} \times_{(\beta_{1},\beta_{2})} S_{2}$ consists of those points having a neighborhood not isomorphic to the unit disc. We observe that the fiber product has the structure of a singular Riemann surface (see Section \ref{singulardef} and Proposition \ref{propo1}).
The space $$S_{1} \times_{(\beta_{1},\beta_{2})}S_{2}-{\rm Sing}$$consists of a  collection of connected Riemann surfaces, say $\widetilde{R}_{k},$ and the map \eqref{lente} restricts to a not necessarily surjective holomorphic map \begin{equation} \label{be}\beta:\widetilde{R}_{k} \to S_{0}\end{equation}(some common branch value of $\beta_{1}$ and $\beta_{2}$ may not be in the image of this map).
Each $\widetilde{R}_{k}$ has a collection of punctures, associated to the points in ${\rm Sing}$, and by filling in these points, 
we obtain a unique, up to biholomorphism, Riemann surface $R_{k}$, called an {\it irreducible component} of the fiber product, and a surjective holomorphic map $\beta:R_{k} \to S_{0}$ extending \eqref{be}.  If both maps $\beta_{1}$ and $\beta_{2}$ have finite degree, then we observe that the number of  irreducible components of the fiber product is bounded above by the greatest common divisor of these two degrees (see Proposition \ref{cota}). Example \ref{ejemplo7} shows that the aforementioned bound is attained.

Now, if we set $S_{1}^{*}=S_{1}-\beta_{1}^{-1}(B)$,  $S_{2}^{*}=S_{2}-\beta_{2}^{-1}(B)$, $S_{0}^{*}=S_{0}-B$, where $B$ is the set consisting of all branch values of either $\beta_{1}$ or $\beta_{2}$, then $\beta_{j}:S_{j}^{*} \to S_{0}^{*}$ is a surjective holomorphic unbranched map and we have
$$S_{1}^{*} \times_{(\beta_{1},\beta_{2})} S_{2}^{*} =\bigcup_{k} R_{k}^{*} \subset S_{1} \times_{(\beta_{1},\beta_{2})} S_{2}-{\rm Sing}=\bigcup_{k} \widetilde{R}_{k} \subset \bigcup_{k} R_{k},$$
where $R_{k}^{*}$ is a connected Riemann surface, this being the complement in $\widetilde{R}_{k}$ of a discrete set of points. 
In Section \ref{Sec:Kleinian} we  provide a Fuchsian group description of the connected components $R_{k}^{*}$ (see Lemma \ref{minimalcomponent}) and we observe that,
if one of the maps $\beta_{j}$ is a regular covering (that is, the quotient map defined by the action of a finite group of conformal automorphisms), then all the irreducible components $R_{k}$ are pairwise isomorphic Riemann surfaces (Corollary \ref{isomorfo}).

This last result is related to the following one. Under the assumptions that $S_{0}$, $S_{1}$ and $S_{2}$ are compact Riemann surfaces, in \cite{H:fiberproduct}, the first author observed that any two irreducible components (after desingularization) of lowest genus are isomorphic, except possibly if their common genus equals $1$, in which case they are still isogenous (this result is not longer true if we allow the algebraic curves to be reducible).

Assume that $S_0=\widehat{\mathbb C}$ and both $S_{1}$ and $S_{2}$ are compact. Then in \cite{Fulton}, Fulton and Hansen proved that the fiber product is connected. This connectedness property may fail if $S_{0}$ has positive genus (see Examples \ref{ejemplo3-1} and \ref{ejemplo3-2}). We also mention that, in \cite{Pakovich}, Pakovich
used the fiber product to study solutions of equations of the form $f(p(x))=g(q(x))$, where $f,g \in {\mathbb C}(z)$ and $p,q:S \to \mathbb{C}$ are meromorphic maps (this equation is related to prime decomposition of rational maps; see also the paper \cite{Fried} for the polynomial case). 

In addition, under the assumptions that $S_0=\widehat{\mathbb C}$ and both $S_{1}$ and $S_{2}$ are compact, in \cite{Pakovich}, Pakovich provided a sufficient condition for the fiber product to be irreducible; this condition being that the degrees of the maps $\beta_{1}$ and $\beta_{2}$ are relatively prime. In the same paper (and under the same assumptions)  necessary and sufficient conditions for irreducibility were provided in terms of the uniformizing Fuchsian groups. By adapting Pakovich's arguments, in the case that the fiber product is connected (we allow $S_0$ to have positive genus or being non-compact and we do not impose compactness assumptions on $S_{1}$ nor $S_{2}$), we provide two conditions, each one implying the irreducibility of the fiber product (see Theorem \ref{teo1}). 

In Section \ref{Sec:fieldofmoduli} we restrict to $S_{0}=\widehat{\mathbb C}$ and both $S_1$ and $S_2$ to be compact (so we may assume all these surfaces to be described by smooth complex algebraic curves, and the maps $\beta_{j}$ to be rational maps over ${\mathbb C}$). Associated to each pair $(S_{j},\beta_{j})$ is its {\it field of moduli}, this being the intersection of all its fields of definition.  We describe an algebraic invariant of the fiber product pair $(S_{1} \times_{(\beta_{1},\beta_{2})} S_{2},\beta)$, called its {\it strong field of moduli}, and we observe that this field is the smallest field containing the fields of moduli of the two corresponding pairs $(S_{1},\beta_{1})$ and $(S_{2},\beta_{2})$ (Theorem \ref{modulifield}). In  Example \ref{Ejemplo:strong} we construct an example where the strong field of moduli is a quadratic extension of the field of moduli.

Finally, in Section \ref{Sec:jacobiana}, assuming the fiber product to be a compact connected Riemann surface, we provide an isogenous decomposition of its Jacobian variety (Theorem \ref{yo}). 
 
Our examples of compact Riemann surfaces will be given by irreducible non-singular projective algebraic curves. In this case, each irreducible component of  the fiber product will be an irreducible projective algebraic curve. If an irreducible component is non-singular, then it represents a compact Riemann surface. Otherwise,  by deleting its singular point we obtain an affine non-singular curve representing an analytically finite Riemann surface, which (after adding the punctures) will represent a compact Riemann surface. Some of the explicit computations have been carried out with the help of  the computational software MAGMA \cite{MAGMA}.

\section{The fiber product of Riemann surfaces}\label{Sec:FiberProduct}
\subsection{The fiber product: the topological setting}\label{Sec:productofibrado}
Let us fix three topological spaces $X_{0}$, $X_{1}$ and $X_{2}$ and two surjective continuous maps
$\beta_{1}:X_{1} \to X_{0}$ and  $\beta_{2}:X_{2} \to X_{0}$.  The {\it fiber product} associated to the pairs $(X_{1},\beta_{1})$ and $(X_{2},\beta_{2})$ is defined as the set
\begin{equation} \label{ast} 
X_{1} \times_{(\beta_{1},\beta_{2})} X_{2}=\{(x_{1},x_{2}) \in X_{1} \times X_{2}: \beta_{1}(x_{1})=\beta_{2}(x_{2})\},
\end{equation}
endowed with the topology induced by the product topology of $X_{1} \times X_{2}.$
There is associated a natural continuous map $$\beta:X_{1} \times_{(\beta_{1},\beta_{2})} X_{2} \to X_{0}$$
such that $\beta=\beta_{1} \circ \pi_{1}=\beta_{2} \circ \pi_{2}$, where $\pi_{j}:X_{1} \times_{(\beta_{1},\beta_{2})} X_{2} \to X_{j}$ is the projection $\pi_{j}(x_{1},x_{2})=x_{j}$ for $j=1,2$.

The fiber product enjoys the following universal property. If $Y$ is a topological space and, for $j=1,2$, there is a  continuous map
$p_{j}:Y \to X_{j}$  such that $\beta_{1} \circ p_{1} = \beta_{2} \circ p_{2}$, then there exists a unique continuous map $h:Y \to X_{1} \times_{(\beta_{1},\beta_{2})} X_{2}$ such that $p_{j}=\pi_{j} \circ h$ (which is given by $h(y)=(p_{1}(y),p_{2}(y))$, for $y \in Y$). The fiber product is, up to homeomorphisms, the unique topological space satisfying the above property.

\subsection{Singular Riemann surfaces}\label{singulardef}
Set ${\mathbb D}=\{z \in {\mathbb C}: |z|<1\}$, the unit disc.

A {\it singular Riemann surface} is defined to be a (not necessarily connected) one-dimensional complex analytic space $X,$ where each point $P \in X$ has a neighborhood  holomorphically equivalent to a set of the form 
$$V_{n,m}:=\{(z,w) \in {\mathbb D}^{2}: z^{n}=w^{m}\},$$for some integers $n,m \ge 1.$  

If $n=1$ or $m=1$, then $V_{n,m}$ is holomorphically equivalent to ${\mathbb D}$. Now, if $n, m \geq 2$, then let 
$d \geq 1$ be the greatest common divisor of $n$ and $m$, and write $n=d\widehat{n}$ and  $m=d\widehat{m}$ (so $\widehat{n}, \widehat{m} \geq 1$ are relatively prime integers). If $\omega$ is a $d$-th primitive root of unity, then we may see that 
$$V_{n,m}=\left\{(z,w) \in {\mathbb D}^{2}: \Pi_{k=0}^{d-1} (z^{\widehat{n}}-\omega^k w^{\widehat{m}})=0\right\},$$
which is homeomorphic to the union of $d$ copies of ${\mathbb D}$ glued at their centers. In particular, if $d=1$, then again $V_{n,m}$ is holomorphically equivalent to ${\mathbb D}$. 
If $d \geq 2$, then the point $P$ is called {\it singular}, and the locus of singular points of $X$, denoted by ${\rm Sing}_{X}$ (or just by ${\rm Sing}$ if the context is clear) is a discrete subset of $X$. It follows that each connected component $Y$ of $X-{\rm Sing}_{X}$ has the structure of a Riemann surface, and the points in $\mbox{Sing}_{X}$ define punctures on it. By adding these punctures (coming from the singular points), we obtain another Riemann surface $Z$, containing $Y$, called an {\it irreducible component} of $X$. If $X$ has only one irreducible component then it is called {\it irreducible}; otherwise, it is called {\it reducible}.

\begin{rema}
If $X$ is a compact singular Riemann surface, then ${\rm Sing}_{X}$ is a finite set, $Z$ is a compact Riemann surface and $Y$ is an analytically finite Riemann surface. Note that, in this case, if for each singular point we have $d=2$, then $X$ is a stable Riemann surface (see, for example, the Bers' definition in \cite{Bers}). 
\end{rema}

Let us consider two singular Riemann surfaces $X$ and $Y$. By an {\it isomorphism} between $X$ and $Y$ we mean a homeomorphism $F:X \to Y$ such that $F({\rm Sing}_{X})={\rm Sing}_{Y}$ and the restriction $F:X-{\rm Sing}_{X} \to Y-{\rm Sing}_{Y}$ is holomorphic.

\subsection{Fiber product of Riemann surfaces}\label{fiberproductdef}
In this section we restrict to the fiber product of Riemann surfaces and holomorphic maps.
Let us fix three connected Riemann surfaces $S_{0}$, $S_{1}$ and $S_{2}$ together with two surjective holomorphic maps
$\beta_{1}:S_{1} \to S_{0}$ and $\beta_{2}:S_{2} \to S_{0}$. Consider the corresponding fiber product
\begin{equation} \label{ast} 
S_{1} \times_{(\beta_{1},\beta_{2})} S_{2}=\{(z_{1},z_{2}) \in S_{1} \times S_{2}: \beta_{1}(z_{1})=\beta_{2}(z_{2})\}
\end{equation}
and its associated map $\beta:S_{1} \times_{(\beta_{1},\beta_{2})} S_{2} \to S_{0}$ such that $\beta=\beta_{1} \circ \pi_{1}=\beta_{2} \circ \pi_{2}$, where $\pi_{j}:S_{1} \times_{(\beta_{1},\beta_{2})} S_{2} \to S_{j}$ is the projection $\pi_{j}(z_{1},z_{2})=z_{j}$.

\begin{prop}\label{propo1}
The fiber product $S_{1}\times_{(\beta_{1},\beta_{2})} S_{2}$ is a singular Riemann surface. Let  $P^{0}=(z^{0}_{1},z^{0}_{2}) \in S_{1}\times_{(\beta_{1},\beta_{2})} S_{2}$, let $n_{j}$ be the local degree of $\beta_{j}$ at $z^{0}_{j}$ (for $j=1,2$) and let $d$ be the greatest common divisor of $n_{1}$ and $n_{2}$.
Then $P^{0}$ is a singular point if and only if $d \geq 2$, in which case, it has a neighborhood of the form $V_{n_{1},n_{2}}$.
\end{prop}
\begin{proof} 
We start the proof by observing that the fiber product inherits structure of a one-dimensional complex space as a subset of the complex surface $S_{1} \times S_{2}$. 
Let $P^{0}=(z^{0}_{1},z^{0}_{2}) \in S_{1}\times_{(\beta_{1},\beta_{2})} S_{2}$ and assume that $\beta_{1}$ has local degree $n_{1} \geq 1$ and that $\beta_{2}$ has local degree $n_{2} \geq 1$ at $z_{2}$. Consider local coordinates $z_{j}:U_{j} \subset S_{j} \to {\mathbb D}$,  $z_{j}(z_{j}^{0})=0$, such that $\beta_{j}(z_{j})=z_{j}^{n_{j}}$, and therefore a neighborhood of $P^0$ can be identified with 
$V_{n_{1},n_{2}}=\{(z_{1},z_{2}) \in {\mathbb D}^{2}: z_{1}^{n_{1}}=z_{2}^{n_{2}}\} \subset {\mathbb C}^{2},$ where $P^0$ is identified with $(0,0)$, and $\beta|_{V_{n_{1},n_{2}}}(z_{1},z_{2})=z_{1}^{n_{1}}=z_{2}^{n_{2}}$. 
It follows that the fiber product is a singular Riemann surface and that $P^{0}$ is a singular point if and only if $d \geq 2,$ where $d$ greatest common divisor of $n_1$ and $n_2.$ 
\end{proof}

\begin{rema}\label{obs2-2}
(1) If, for $j=1,2$, the map $\beta_{j}$ is a regular branched cover with deck group $G_{j}<{\rm Aut}(S_{j})$, then the map $\beta$ is the quotient map given by the action of the direct product $G_{1} \times G_{2}$ on the fiber product.
(2) If $S_{1}$ and $S_{2}$ are both compact Riemann surfaces (so it is $S_{0}$), then the fiber product (a closed subset of the compact complex surface $S_{1} \times S_{2}$) is compact. 
\end{rema}

\begin{rema}[A universal property]\label{universal}
As a consequence of the universal property of the fiber product, if $Y$ is a Riemann surface and, for $j=1,2$, there is a holomorphic map
$p_{j}:Y \to S_{j}$  such that $\beta_{1} \circ p_{1} = \beta_{2} \circ p_{2}$, then there exists a unique holomorphic map $h:Y \to S_{1} \times_{(\beta_{1},\beta_{2})} S_{2}$ such that $p_{j}=\pi_{j} \circ h$. 
 \end{rema}

\subsection{On the irreducible components}
As stated in the introduction of the paper, we have that 
$$S_{1} \times_{(\beta_{1},\beta_{2})}S_{2}-{\rm Sing}=\bigcup_{k} \widetilde{R}_{k} \subset \bigcup_{k} R_{k},$$ 
where $\widetilde{R}_{k} \subset R_{k}$, each $R_{k}$ is a connected Riemann surface and   
$R_{k}-\widetilde{R}_{k}$ is the collection of punctures coming from the points in ${\rm Sing}$.

Let $B \subset S_{0}$ be the discrete subset consisting of the union of the branch values of the maps $\beta_{1}$ and $\beta_{2}$. Set 
$S_{0}^{*}=S_{0}-B$, $S_{1}^{*}=S_{1}-\beta_{1}^{-1}(B)$ and $S_{2}^{*}=S_{2}-\beta_{2}^{-1}(B)$.
The restriction of $\beta_{j}$ from $S_{j}^{*}$ to $S_{0}^{*}$, which will be still denoted with the same symbol, is an unbranched holomorphic surjective map. We may now consider the fiber product of the new pairs $(S_{1}^{*},\beta_{1})$ and $(S_{2}^{*},\beta_{2})$. It is not difficult to see that $\mbox{Sing} \subset \beta^{-1}(B)$, that  $S_{1}^{*} \times_{(\beta_{1},\beta_{2})} S_{2}^{*}=(S_{1} \times_{(\beta_{1},\beta_{2})} S_{2})-\beta^{-1}(B)$, and
$$S_{1}^{*} \times_{(\beta_{1},\beta_{2})} S_{2}^{*}= \bigcup_{k \in J} R^{*}_{k},$$ 
where $R^{*}_{k} \subset \widetilde{R}_{k}$ and 
$\widetilde{R}_{k}-R^{*}_{k}$ consists of those points $(z_{1},z_{2}) \in \widetilde{R}_{k}$ for which at least one of the coordinates $z_{j}$ is the preimage of a branch value of $\beta_{j}$. The map $\beta$ restricts to a surjective holomorphic map  $\beta: S_{1}^{*} \times_{(\beta_{1},\beta_{2})} S_{2}^{*} \to S_{0}^{*}$, and the restrictions $\pi_j:R^{*}_{k} \to S_{j}^{*}$, $j=1,2$, satisfies that $\beta_{1} \circ \pi_{1}=\beta_{2} \circ \pi_{2}=\beta$.

A consequence of the universal property of the fiber product (see Remark \ref{universal}), is the following fact.

\begin{lemm}\label{lemita}
If $Y$ is a connected Riemann surface and $p_{j}:Y \to S_{j}^{*}$ are surjective holomorphic maps so that $\beta_{1} \circ p_{1}=\beta_{2} \circ p_{2}$, then there exists  an irreducible component $R^{*}_{k}$ and a holomorphic map $h:Y \to R^{*}_{k}$  so that $p_{j}=\pi_{j} \circ h$. 
\end{lemm}

The following result provides an upper bound on the number of irreducible components of the fiber product in the case that both holomorphic maps $\beta_{j}$ are of finite degree.

\begin{prop}\label{cota}
If $\beta_{1}$ and $\beta_{2}$ both have finite degree, then 
the number of irreducible components of the fiber product of the two pairs $(S_{1},\beta_{1})$ and $(S_{2},\beta_{2})$
is at most the greatest common divisor of the degrees of $\beta_1$ and $\beta_2$.
\end{prop}
\begin{proof}
Keeping the previous notations, each of the covers $\beta_{j}:S^{*}_{j} \to S^{*}_{0}$ is unbranched of degree $d_{j}$. Let
$G$ be the fundamental group of $S^{*}_{0}$. Then these covers correspond to $G$-sets $E_{1}$
and $E_{2}$ of cardinality $d_{1}$ and $d_{2}$, on which the group $G$ acts transitively. We
get subgroups $H_{1}$ and $H_{2}$ of $G$ by taking the stabilizer of chosen elements $e_{1} \in E_{1}$
and $e_{2} \in E_{2}$, which remain well-defined up to conjugacy if these
elements are chosen differently.
Using the categorical definition of the fiber product, one shows that the fiber product
$S_{1}^{*} \times_{(\beta_{1},\beta_{2})} S_{2}^{*}$ corresponds to the $G$-set $E_{1} \times E_{2}$. This $G$-set may not be
transitive any longer. It is if and only if the fiber product is irreducible.
This happens if and only if the stabilizer of some (hence any) point is of
index $d_{1} d_{2}$ in G, that is, if and only if $H_{1} \cap H_{2}$ has index $d_{1} d_{2}$
in $G$.
In any case, the stabilizers involved have index divisible by both $d_{1}$ and $d_{2}$,
since they are contained in conjugates of both $H_{1}$ and $H_{2}$. Therefore these
indices are all at least ${\rm lcm} (d_1, d_2)$. This implies that the number of orbits
is at most $d_{1} d_{2} / {\rm lcm} (d_1, d_2) = \gcd (d_1, d_2)$. But this number of orbits
is nothing but the number of irreducible components of the fiber product.
\end{proof}

\begin{example}[{\bf Fiber product of two Fermat curves}]\label{ejemplo7}
Let $m,n \ge 1$ be integers and consider
$$\begin{array}{c}
\beta_{1}:S_{1}=\{[x_{1}:x_{2}:x_{3}] \in {\mathbb P}_{\mathbb C}^{2}: x_{1}^{n}+x_{2}^{n}+x_{3}^{n}=0\} \to \widehat{\mathbb C}: [x_{1}:x_{2}:x_{3}] \mapsto -( \tfrac{x_{2}}{x_{1}} )^{n},\\
\beta_{2}:S_{2}=\{[y_{1}:y_{2}:y_{3}] \in {\mathbb P}_{\mathbb C}^{2}: y_{1}^{m}+y_{2}^{m}+y_{3}^{m}=0\}\to \widehat{\mathbb C}:[y_{1}:y_{2}:y_{3}] \mapsto -( \tfrac{y_{2}}{y_{1}})^{m}.
\end{array}
$$ 

Note that $\beta_{1}$ (respectively, $\beta_{2}$) is a branched regular covering map whose deck group is the abelian group ${\mathbb Z}_{n}^{2}$ (respectively, ${\mathbb Z}_{m}^{2}$).  The fiber product $S_{1} \times_{(\beta_{1},\beta_{2})} S_{2}$ is represented by the algebraic curve in  ${\mathbb P}_{\mathbb C}^{2} \times {\mathbb P}_{\mathbb C}^{2}$ with coordinates $([x_{1}:x_{2}:x_{3}],[y_{1}:y_{2}:y_{3}])$ given by the equations
$$
x_{1}^{n}+x_{2}^{n}+x_{3}^{n}=0, \, y_{1}^{m}+y_{2}^{m}+y_{3}^{m}=0, \; x_{2}^{n} y_{1}^{m}=x_{1}^{n} y_{2}^{m}.$$

If $d=\gcd(n,m)$, then Proposition \ref{cota} asserts that the number of such irreducible components should be at most $d^{2}$. We claim that in this example such a number is exactly $d^{2}$. In fact, if $n=d n_{1}$ and $m=d m_{1}$ and $\omega_{d}$ is a primitive $d$-root of unity, then an affine model $S$, by taking $x_{1}=y_{1}=1$, is given by
$$\begin{array}{lll}
S&=& \{(x_{2},x_{3},y_{2},y_{3}) \in {\mathbb C}^{4}: 1+x_{2}^{n}+x_{3}^{n}=0, 1+y_{2}^{m}+y_{3}^{m}=0, x_{2}^{n}=y_{2}^{m}\}\\
&=&\{(x_{2},x_{3},y_{2},y_{3}) \in {\mathbb C}^{4}: 1+x_{2}^{n}+x_{3}^{n}=0, x_{3}^{n}=y_{3}^{m} , x_{2}^{n}=y_{2}^{m}\}\\
&=&\bigcup_{r,s=0}^{d-1} S_{r,s},
\end{array}
$$
where
$$S_{r,s}=\{(x_{2},x_{3},y_{2},y_{3}) \in {\mathbb C}^{4}: 1+x_{2}^{n}+x_{3}^{n}=0, x_{3}^{n_{1}}=\omega_{d}^{r}y_{3}^{m_{1}} , x_{2}^{n_{1}}=\omega_{d}^{s} y_{2}^{m_{1}}\}.$$

Note that, by setting $Y_{3}:=\omega_{d}^{r/m_{1}}y_{3}$ and $Y_{2}:=\omega_{d}^{s/m_{1}}y_{2}$, we may see that $S_{r,s}$ is isomorphic to $S_{0,0}$ (see Corollary \ref{isomorfo}).
As $n_{1}$ and $m_{1}$ are relatively prime, for each $(x_{2},x_{3},y_{2},y_{3}) \in S_{0,0}$, there is a pair $(t_{2},t_{3}) \in {\mathbb C}^{2}$ so that $x_{2}=t_{2}^{m_{1}}$, $y_{2}=t_{2}^{n_{1}}$, $x_{3}=t_{3}^{m_{1}}$ and $y_{3}=t_{3}^{n_{1}}$; moreover, $1+t_{2}^{dn_{1}m_{1}}+t_{3}^{dn_{1}m_{1}}=0$. This permits to see that $S_{0,0}$ (so each $S_{r,s}$) is isomorphic to the Fermat curve (a compact Riemann surface)
$$\{[w_{1}:w_{2}:w_{3}] \in {\mathbb P}_{\mathbb C}^{2}: w_{1}^{dn_{1}m_{1}}+w_{2}^{dn_{1}m_{1}}+w_{3}^{dn_{1}m_{1}}=0\}.$$
\end{example}

\subsection{A Fuchsian group description of the connected components $R^{*}_{k}$}\label{Sec:Kleinian}
Let us continue with the previous notations.
Let us assume that $S_{0}^{*}$ (so $S_{1}^{*}$ and $S_{2}^{*}$) is a hyperbolic Riemann surface. Then there is a Fuchsian group $\Gamma_{0}$ acting on the hyperbolic plane ${\mathbb H}^{2}$ such that $S_{0}^{*}$ is conformally equivalent to ${\mathbb H}^{2}/\Gamma_{0}$. As a consequence of covering theory (see, for example, the book  \cite{Massey}), there is a finite index subgroup $\Gamma_{j}$ of $\Gamma_{0}$ so that ${\mathbb H}^{2}/\Gamma_{j}$ is isomorphic as Riemann surface to $S_{j}^{*}$ and the covering $\beta_{j}:S_{j}^{*} \to S_{0}^{*}$ is realized by the inclusion of $\Gamma_{j}$ in $\Gamma_{0}$.

\begin{rema}\label{obs:isofiber1}
\begin{enumerate} \mbox{}
\item The non-hyperbolic situation can be carried out in a similar way, by replacing $\mathbb{H}^2$ by either the complex plane of the Riemann sphere.
\item The choice of the Fuchsian group  $\Gamma_{j}$ above is unique up to conjugation by an element of $\Gamma_{0}$, that is, we may replace $\Gamma_{j}$ by $\gamma \Gamma_{j} \gamma^{-1}$, for each $\gamma \in \Gamma_{0}$.
\item For $j=1,2$, let $T_{j}^{*}$ be a connected Riemann surface, $\delta_{j}:T_{j}^{*} \to S_{0}^{*}$ be an unbranched surjective holomorphic map, and assume that there exists an isomorphism $\phi_{j}:S_{j}^{*} \to T_{j}^{*} $  such that $\beta_{j} = \delta_{j} \circ \phi_{j}$. Then $(\phi_{1},\phi_{2})$ induces an isomorphism between the fiber products $S_{1}^{*} \times_{(\beta_{1},\beta_{2})}S_{2}^{*}$ and $T_{1}^{*} \times_{(\delta_{1},\delta_{2})}T_{2}^{*}$.
\end{enumerate}
\end{rema}

\begin{lemm}\label{minimalcomponent}
With the same notations as before, the following statements are equivalent:
\begin{enumerate}
\item $X$ is a connected component of the fiber product $S_{1}^{*} \times_{(\beta_{1},\beta_{2})}S_{2}^{*}$.
\item $X$ is isomorphic to the quotient $\mathbb{H}/K$ where $K=\gamma_{1} \Gamma_{1} \gamma_{1}^{-1} \cap \gamma_{2} \Gamma_{2} \gamma_{2}^{-1}$ for some $\gamma_{1}, \gamma_{2} \in \Gamma_{0}.$
\end{enumerate}
\end{lemm}

\begin{proof}
Let us assume that $X$ is isomorphic to $\mathbb{H}/K$ with $K$ as in (2).  By Remark \ref{obs:isofiber1} we can assume, without loss of generality, that $K=\Gamma_{1} \cap \Gamma_{2}.$ Clearly, the inclusion of $K=\Gamma_1 \cap \Gamma_2$ in $\Gamma_{j}$ induces a holomorphic map $p_{j}:X \to S_{j}$ such that $\beta_{1} \circ p_{1}=\beta_{2} \circ p_{2}$. By Lemma \ref{lemita}, there exists a connected component $R$ of the fiber product and a holomorphic map $h:X \to R$  so that $p_{j}=\pi_{j} \circ h$. But this implies that $R$ is isomorphic to a quotient of the form ${\mathbb H}^{2}/\Gamma'$, where $\Gamma'$ is a Fuchsian group which contains $K=\Gamma_1 \cap \Gamma_2$ and it is contained in $\Gamma_{1}$ and in $\Gamma_{2}$. This asserts that $K=\Gamma'$ and therefore $X(=R)$ is a connected component of the fiber product, as desired.

Conversely, let $X$ be a connected component of the fiber product $S_{1}^{*} \times_{(\beta_{1},\beta_{2})}S^{*}_{2}$. As we have the restriction $\pi_{j}:X \to S^{*}_{j}$ such that $\beta_{1} \circ \pi_{1}=\beta_{2} \circ \pi_{2}$, there must be a subgroup $\Gamma_{X} \leq \Gamma_{1} \cap \Gamma_{2}$ such that $X={\mathbb H}^{2}/\Gamma_{X}$. It follows that there is covering $h:X \to R^{*}={\mathbb H}^{2}/(\Gamma_{1} \cap \Gamma_{2})$ such that $\pi_{j} \circ h=\pi_{j}$ (see also  Lemma \ref{lemita}).  In particular, if $X$ is not of the claimed form, then we may delete it from the fiber product in order to produce a contradiction to the universal property (at the level of topological spaces) of the fiber product.
\end{proof}

\begin{rema}\label{obs:isofiber3} Observe that for $\gamma_{1}, \gamma_{2} \in \Gamma_{0}$, the Fuchsian groups $$\gamma_{1} \Gamma_{1} \gamma_{1}^{-1} \cap \gamma_{2} \Gamma_{2} \gamma_{2}^{-1} \mbox{ and }\Gamma_{1} \cap \gamma_{3} \Gamma_{2} \gamma_{3}^{-1},$$ where $\gamma_{3}=\gamma_{1}^{-1} \circ \gamma_{2}$, are conjugate by $\gamma_{1}$. This shows that each connected component of $S_{1}^{*} \times_{(\beta_{1},\beta_{2})}S_{2}^{*}$ is isomorphic to a quotient of the form ${\mathbb H}^{2}/(\Gamma_{1} \cap \gamma \Gamma_{2} \gamma^{-1})$ for a suitable $\gamma \in \Gamma_0.$ 
\end{rema}

\begin{coro}\label{isomorfo} 
If one of the maps $\beta_{j}:S_{j} \to S_{0}$ is a regular (branched) cover, then all connected components of the fiber product $S^{*}_{1} \times_{(\beta_{1},\beta_{2})} S^{*}_{2}$ are isomorphic. In particular, all irreducible components of $S_{1} \times_{(\beta_{1},\beta_{2})} S_{2}$ are isomorphic.

\end{coro}
\begin{proof}
If $\beta_{2}:S_{2} \to S_{0}$ is a regular (branched) cover, then $\beta_{2}:S^{*}_{2} \to S^{*}_{0}$ is a regular cover. Then, 
$\Gamma_{2}$ is a normal subgroup of $\Gamma_{0}$. Thus the result follows from Lemma \ref{minimalcomponent} and Remark \ref{obs:isofiber3}.
\end{proof}

\begin{rema}\label{obs:isofiber4}
In Example \ref{2componentes} we will provide a fiber product $S^{*}_{1} \times_{(\beta_{1},\beta_{2})}S_{2}^{*}$ with exactly two components, one is a punctured sphere and the other is a punctured torus. In this example, $S_{0}^{*}=\widehat{\mathbb C}-\{\infty,-1, (37 \pm 45 i\sqrt{15})/512\}$, and $S_{1}^{*}=S_{2}^{*}$ is a seven-punctured sphere. Inside $\Gamma_{0}$ (a free group of rank three) we may choose $\Gamma_{1}=\Gamma_{2}$ (a free group of rank nine). As $\beta_{1}$ is not a regular covering, the group $\Gamma_{1}$ is not a normal subgroup of $\Gamma_{0}$. In this case, if we take $\gamma \in \Gamma_{0}$ such that $\gamma \Gamma_{1} \gamma^{-1} \neq \Gamma_{1}$, then we obtain the components ${\mathbb H}^{2}/\Gamma_{1}$ and ${\mathbb H}^{2}/(\Gamma_{1} \cap \gamma \Gamma_{1} \gamma^{-1})$. The first one corresponds to the genus zero surface (which is isomorphic to $S_{0}^{*}$) and the second one is the genus one component.
\end{rema}

\section{On the connectedness and irreducibility of the fiber product}\label{Sec:irreducible}
In this section we study the connectedness and irreducibility of the fiber product.

\subsection{Connectedness in the compact situation}
Let us assume that all surfaces $S_{0}$, $S_{1}$ and $S_{2}$ are compact. In particular, the maps $\beta_{1}$ and $\beta_{2}$ will have finite degree. So, as previously observed, the number of irreducible components is at most the greatest common divisor of the two degrees.
In \cite{Fulton} it was observed that the fiber product is connected for $S_{0}=\widehat{\mathbb C}$.
The following examples show that the fiber product might not be connected in the case that $S_{0}$ has positive genus.

\begin{example}{[{\bf non-connected fiber product when $S_{0}$ has genus at least two}]}\label{ejemplo3-1}
Let us consider compact Riemann surfaces $S_{0}$ and $S$, both of genus at least two, and let  $\pi:S \to S_{0}$ be an unbranched covering map of degree $d \geq 2$. Take $S_{1}=S_{2}=S$, $\beta_{1}=\beta_{2}=\pi$.  Under the above conditions, it is possible to check that  the fiber product has no singularities. In particular, any irreducible component is a connected Riemann surface, and any two different irreducible components must be disjoint.
If $S_{1} \times_{(\beta_{1},\beta_{2})} S_{2}$ were connected, then (as $\beta$ has degree $d^{2}$) the genus of $S_{1} \times_{(\beta_{1},\beta_{2})} S_{2}$ would be strictly bigger than of $S$. Now, by taking $P_{j}:S \to S_{j}$ equal to the identity we see that $S$ covers $S_{1} \times_{(\beta_{1},\beta_{2})} S_{2}$, providing a contradiction to the Riemann-Hurwitz formula (since the genus of $S$ is strictly less than the genus of $S_{1} \times_{(\beta_{1},\beta_{2})} S_{2}$). 
\end{example}

\begin{example}{[{\bf non-connected fiber product when $S_{0}$ has genus one}]}\label{ejemplo3-2}
 For $\lambda \in {\mathbb C}-\{0,1\}$, let us consider the Riemann surfaces 
 $$\begin{array}{l}
S=\{[x_{1}:x_{2}:x_{3}:x_{4}] \in {\mathbb P}^{3}_{\mathbb C}: x_{1}^{2}+x_{2}^{2}+x_{3}^{2}=0, \; \lambda x_{1}^{2}+x_{2}^{2}+x_{4}^{2}=0\},\\
S_{0}=\{[y_{1}:\ldots :y_{5}] \in {\mathbb P}^{4}_{\mathbb C}: y_{3}^{2}=y_{1} y_{2}, \; y_{5}(y_{1}+y_{2})+y_{4}^{2}= \lambda y_{1}+y_{2}+y_{5}=0\}.
\end{array}
$$

The map $\pi:S \to S_{0}:[x_{1}:x_{2}:x_{3}:x_{4}] \mapsto [x_{1}^{2}:x_{2}^{2}:x_{1}x_{2}:x_{3}x_{4}:x_{4}^{2}]$
is an unbranched two-fold cover whose deck group is cyclic generated by the involution
$\tau([x_{1}:x_{2}:x_{3}:x_{4}])=[-x_{1}:-x_{2}:x_{3}:x_{4}].$ 
If  $S_{1}=S_{2}=S$ and $\beta_{1}=\beta_{2}=\pi$, then the fiber product in this case is
$$\begin{array}{l}
X=\left\{\left([x_{1}:x_{2}:x_{3}:x_{4}],[z_{1}:z_{2}:z_{3}:z_{4}]\right) \in \left({\mathbb P}^{3}_{\mathbb C}\right)^{2} : 
x_{1}^{2}+x_{2}^{2}+x_{3}^{2}=0,\right.\\
\left. \; \lambda x_{1}^{2}+x_{2}^{2}+x_{4}^{2}=0, 
z_{1}^{2}+z_{2}^{2}+z_{3}^{2}=0, \; \lambda z_{1}^{2}+z_{2}^{2}+z_{4}^{2}=0, \right. \\
\left. 
[x_{1}^{2}:x_{2}^{2}:x_{1}x_{2}:x_{3}x_{4}:x_{4}^{2}]=[z_{1}^{2}:z_{2}^{2}:z_{1}z_{2}:z_{3}z_{4}:z_{4}^{2}]
\right\}
\end{array}.
$$

The equality
$$[x_{1}^{2}:x_{2}^{2}:x_{1}x_{2}:x_{3}x_{4}:x_{4}^{2}]=[z_{1}^{2}:z_{2}^{2}:z_{1}z_{2}:z_{3}z_{4}:z_{4}^{2}]$$ in ${\mathbb P}^{4}_{\mathbb C}$ yields two possibilities; either 
$[z_{1}:z_{2}:z_{3}:z_{4}]=[x_{1}:x_{2}:x_{3}:x_{4}]$ or $[z_{1}:z_{2}:z_{3}:z_{4}]=[x_{1}:x_{2}:-x_{3}:-x_{4}].$
In this way, $X=A \cup B$, where
$$A=\left\{\left([x_{1}:x_{2}:x_{3}:x_{4}],[x_{1}:x_{2}:x_{3}:x_{4}]\right) \in \left({\mathbb P}^{3}_{\mathbb C}\right)^{2} : \right.$$
$$\left. x_{1}^{2}+x_{2}^{2}+x_{3}^{2}=0, \; \lambda x_{1}^{2}+x_{2}^{2}+x_{4}^{2}=0 \right\}$$
$$B=\left\{\left([x_{1}:x_{2}:x_{3}:x_{4}],[x_{1}:x_{2}:-x_{3}:-x_{4}]\right) \in \left({\mathbb P}^{3}_{\mathbb C}\right)^{2} : \right.$$
$$\left. x_{1}^{2}+x_{2}^{2}+x_{3}^{2}=0, \; \lambda x_{1}^{2}+x_{2}^{2}+x_{4}^{2}=0 \right\}$$

Clearly, $A \cap B=\emptyset$ and $A$ and $B$ are both isomorphic to $S$.
\end{example}

\subsection{On the irreducibility of the fiber product}
The following example, communicated to the first author by Gabino Gonz\'alez-Diez and already in \cite{H:fiberproduct}, shows that, even in the case that the fiber product is connected, it might be reducible.

\begin{example}\label{2componentes}
If $S_{1}=S_{2}=S_{0}=\widehat{\mathbb C}$ and $\beta_{1}(z)=\beta_{2}(z)=z(z^3+z^2+1)$, then an affine algebraic model of $S_{1} \times_{(\beta_1, \beta_2)} S_{2}$ is given by
$$\{(x,y) \in {\mathbb C}^{2}: x(x^{3}+x^{2}+1)=y(y^{3}+y^{2}+1)\}.$$
Since 
$$x(x^{3}+x^{2}+1)-y(y^{3}+y^{2}+1) =  
(x - y) (1 + x^2 + x^3 + x y + x^2 y + y^2 + x y^2 + y^3),$$
we can see that the  fiber product
consists of two irreducible components, one of them of genus zero and the other of genus one. See also Remark \ref{obs:isofiber4}.
\end{example}

The next result states sufficient conditions for the fiber product, when it is connected (for instance, if $S_{0}$ has genus zero), to be irreducible. 
Let us denote by {\it lcm} the least common multiple, and by {\it gcd} the greatest common divisor.

\begin{theo}\label{teo1}
Let $S_{0}$, $S_{1}$ and $S_{2}$ be connected Riemann surfaces, let 
$\beta_{1}:S_{1} \to S_{0}$ and $\beta_{2}:S_{2} \to S_{0}$ be two surjective holomorphic maps so that the fiber product $S_{1} \times_{(\beta_{1},\beta_{2})} S_{2}$ is connected. Assume that, for each $q \in S_{0}$ and each $j \in \{1,2\}$, the local degrees of $\beta_{j}$ at its preimages of $q$ is bounded (this happens, for instance, if the surfaces are compact or more general if the maps are of finite degree) and set  
$$a_{q}^{(j)}:={\rm lcm}\left({\rm ord}_{\beta_{j}}(z) : \beta_{j}(z)=q\right).$$ If either 
\begin{enumerate}
\item $\gcd({\rm deg}(\beta_{1}),  {\rm deg}(\beta_{2}))=1$; or 
\item $\gcd(a_{q}^{(1)},a_{q}^{(2)})=1$, for every $q \in S_{0}$, 
\end{enumerate} then the fiber product $S_{1} \times_{(\beta_{1},\beta_{2})} S_{2}$ is irreducible and, in particular, 
$S_{1} \times_{(\beta_{1},\beta_{2})} S_{2}$ is a connected Riemann surface.

\end{theo}
\begin{proof}
The singular points of $S_{1} \times_{(\beta_{1},\beta_{2})} S_{2}$ are those points $(z_{1},z_{2}) \in S_{1} \times S_{2}$ such that 
$\beta_{1}(z_{1})=\beta_{2}(z_{2})$ and so that, for each $j=1,2$, the point $z_{j} \in S_{j}$ is a critical point of $\beta_{j}$, that is,  ${\rm ord}_{\beta_{j}}(z_{j})=n_{j} \geq 2$. In local coordinates, we may assume that the singular point is $(0,0)$ and that $\beta_{j}(z)=z^{n_{j}}$; so 
a neighborhood of such a singular point looks (locally) like $\{(z,w) \in {\mathbb C}^{2}: z^{n_{1}}=w^{n_{2}}\}$.
In this way, if $\gcd(n_{1},n_{2})=d$, then the singular point $(z_{1},z_{2})$ has a neighborhood that looks like 
$d$ different cones glued along such a point. These are also the points of possible  intersection of two different irreducible components.

Let us assume that (1) holds. Let $R_{k}$ be any of the irreducible components of $S_{1} \times_{(\beta_{1},\beta_{2})} S_{2}$ and let 
$d_{1,k}={\rm deg}(\pi_{1}:R_{k} \to S_{1})$ and $d_{2,k}={\rm deg}(\pi_{2}:R_{k} \to S_{2})$. As $\beta=\beta_{1} \circ \pi_{1}=\beta_{2} \circ \pi_{2}$, it holds that 
$d_{1,k} \cdot {\rm deg}(\beta_{1})=d_{2,k} \cdot {\rm deg}(\beta_{2}) \leq {\rm deg}(\beta)={\rm deg}(\beta_{1}){\rm deg}(\beta_{2})$. In particular, 
$d_{1,k} \leq {\rm deg}(\beta_{2})$ and $d_{2,k} \leq {\rm deg}(\beta_{1})$.  Now, the condition $\gcd({\rm deg}(\beta_{1}),  {\rm deg}(\beta_{2}))=1$ asserts that $d_{j,k}={\rm deg}(\beta_{j})$, that is, the degree of the map $\beta_{j} \circ \pi_{j}:R_{k} \to S_{0}$ coincides with that of $\beta:S_{1} \times_{(\beta_{1},\beta_{2})} S_{2} \to S_{0}$. This ensures that the fiber product has only one irreducible component (the degree of $\beta$ is the sum of the degree of its restrictions to each of the irreducible components).

Let us now assume that (2) holds. Under the hypothesis, each of the singular points of $S_{1} \times_{(\beta_{1},\beta_{2})} S_{2}$ has a neighborhood homeomorphic to a disc.  As already we know that $S_{1} \times_{(\beta_{1},\beta_{2})} S_{2}$ is connected, the result follows. 
\end{proof}

\begin{rema}
We should mention that the case (1) in the previous theorem, for $S_{0}$ of genus zero, was previously obtained by Pakovich in \cite{Pakovich}; our arguments are similar.
\end{rema}

The following example shows that the sufficient conditions in Theorem \ref{teo1} are not necessary ones.

\begin{example}[{\bf the conditions of Theorem \ref{teo1} are not necessary}]\label{ejemplo2}
Let us consider $S_{0}=S_{1}=S_{2}=\widehat{\mathbb C}$, $\beta_{1}(z)=4z^{3}(1-z^{3})$ and $\beta_{2}(w)=-27w^{4}(w^2-1)/4$. 
In this case, ${\deg}(\beta_{1})=6={\rm deg}(\beta_{2})$, so Condition (1) of Theorem \ref{teo1} does not hold.
Also, as $\beta_{1}^{-1}(\infty)=\{\infty\}=\beta_{2}^{-1}(\infty)$ and ${\rm ord}_{\beta_{1}}(\infty)=6={\rm ord}_{\beta_{2}}(\infty)$, neither Condition (2) of Theorem \ref{teo1} holds. In this case, 
$$S_{1} \times_{(\beta_{1},\beta_{2})} S_{2}=\left\{([z:t],[w:s]) \in \mathbb{P}^1_{\mathbb{C}} \times \mathbb{P}^1_{\mathbb{C}}: 16z^{3}(t^{3}-z^{3})s^{6}+27w^{4}(w^{2}-s^{2})t^{6}=0  \right\}$$
has one singular point $p_{1}$ with a neighborhood being $6$ cones glued at their vertices (corresponding to the preimage of $\infty$). Let us consider the irreducible component $R$ containing the point $p_{2}=(0,0)$. If we restrict the corresponding projection maps $\pi_{j}:R \to \widehat{\mathbb C}$, then $\beta_{1} \circ \pi_{1}$ has at $p_{2}$ local degree a multiple of $4$ and $\beta_{2} \circ \pi_{2}$ has at $p_{2}$ local degree a multiple of $3$. As $\beta|_{R}=\beta_{1} \circ \pi_{1}|_{R}=\beta_{2} \circ \pi_{2}|_{R}$, its local degree at $p_{2}$ divides the least common multiple of $4$ and $3$, that is, it divides $12$. As the degree of $\beta$ is $12$, we 
may see that $R$ is the unique component. This an irreducible curve of geometric genus $7$.
\end{example}

\begin{example}[{\bf An example for Theorem \ref{teo1}}]\label{ejemplo1} Let us consider
$$
\begin{array}{c}
\beta_{1}:S_{1}=\left\{[x:y:z] \in {\mathbb P}_{\mathbb C}^{2}: y^{3}-x^{2}z+xz^{2}=0\right\} \to \widehat{\mathbb C}: [x:y:z] \mapsto \frac{x}{z},\\ 
\beta_{2}:S_{2}=\left\{[x_{1}:x_{2}:x_{3}] \in {\mathbb P}_{\mathbb C}^{2}: x_{1}^{2}+x_{2}^{2}+x_{3}^{2}=0\right\} \to \widehat{\mathbb C}: [x_{1}:x_{2}:x_{3}] \mapsto -\left( \frac{x_{2}}{x_{1}} \right)^{2}.
\end{array}
$$

Observe that $S_{1}$ has genus one and $S_{2}$ has genus zero. 
As ${\rm deg}(\beta_{1})=3$, ${\rm deg}(\beta_{2})=4$,  ${\gcd}({\rm deg}(\beta_{1}),{\rm deg}(\beta_{2}))=1$, so Theorem \ref{teo1} asserts that the fiber product is irreducible; in fact it can be represented by the algebraic curve in ${\mathbb P}_{\mathbb C}^{2} \times {\mathbb P}_{\mathbb C}^{2}$ with coordinates $([x:y:z],[x_{1}:x_{2}:x_{3}])$ given by the equations
$$
 x_{1}^{2}+x_{2}^{2}+x_{3}^{2}=0, \, y^{3}-x^{2}z+xz^{2}=0, \; x x_{1}^{2}=-z x_{2}^{2},$$
which is isomorphic to the following irreducible curve of genus $4$:
$$R=\left\{[y:v:w:t]: y^{3}t-v^{4}-v^{2}t^{2}=0, t^{2}+v^{2}+w^{2}=0\right\} \subset {\mathbb P}^{3}_{\mathbb C}.$$

The Riemann surface $R$ has the following automorphisms
$$T([y:v:w:t])=[e^{2 \pi i/3} y:v:w:t],$$
$$A([y:v:w:t])=[y:-v:w:t], \; B([y:v:w:t])=[y:v:-w:t],$$
so that $\langle T,A,B\rangle=\langle T \rangle \times \langle A,B \rangle \cong {\mathbb Z}_{3} \times {\mathbb Z}_{2}^{2}.$ The map $$F:R \to \widehat{\mathbb C}:[y:v:w:t] \mapsto x=\frac{w}{iv-t}$$
provides a regular branched cover with deck group $\langle T \rangle$. The branch values of $F$ are given by the points $\infty$, $0$, $\pm i$, $\pm 1$. It follows that $R$ can be also described by the cyclic trigonal gonal curve
$$y^{3}=x(x^{4}-1).$$

The group $\langle A, B \rangle$, under the map $F$, corresponds in this model to the group $\langle a(x,y)=(1/x, -y/x^{2}), b(x)=(-x,-y)\rangle$. 
\end{example}

\section{The strong field of moduli of the fiber product}\label{Sec:fieldofmoduli}
Throughout this section we shall assume $S_{0}=\widehat{\mathbb C}$, that $S_{1}$ and $S_{2}$ are compact Riemann surfaces defined by irreducible non-singular complex algebraic curves, and that $\beta_{j}$ are rational maps. We recall that under this assumption, the fiber product $S_{1} \times_{(\beta_{1},\beta_{2})} S_{2}$ is connected, and that $S_{0}^{\sigma}=S_{0}$ for every $\sigma \in {\rm Gal}({\mathbb C})$ (the group of field automorphisms of ${\mathbb C}$). 

\subsection{The field of moduli of pairs}
Let us consider a pair $(R,\eta)$, where $R$ is a compact Riemann surface and $\eta:R \to \widehat{\mathbb C}$ is a non-constant meromorphic map. As a consequence of the Riemann-Roch Theorem \cite{Roch}, we may assume that $R$ is a non-singular complex projective algebraic curve, say  given as the common zeroes of the homogeneous polynomials $P_{1},\ldots,P_{n}$, and that $\eta=Q_{1}/Q_{2}$, where $Q_{1}$ and $Q_{2}$ are homogeneous polynomials of the same degree.
If $\sigma \in {\rm Gal}({\mathbb C})$, then we denote by $P_{j}^{\sigma}$ (respectively, $Q_{j}^{\sigma}$) the polynomial obtained from $P_{j}$ (respectively, $Q_{j}$) by applying $\sigma$ to its coefficients. The polynomials $P_{1}^{\sigma},\ldots,P_{n}^{\sigma}$ define a non-singular complex projective algebraic curve $R^{\sigma}$ (homeomorphic to $R$) and $\eta^{\sigma}=Q_{1}^{\sigma}/Q_{2}^{\sigma}$ is a rational map on it.
We say that $(R^{\sigma},\eta^{\sigma})$ is isomorphic to $(R,\eta)$ if there is an isomorphism $f_{\sigma}:R \to R^{\sigma}$ so that $\eta^{\sigma} \circ f_{\sigma}=\eta$; we denote it by  $(R^{\sigma},\eta^{\sigma}) \equiv (R,\eta)$. The {\it field of moduli} of the pair $(R,\eta)$ is defined as the fixed field of the group 
$$G=\left\{\sigma \in {\rm Gal}({\mathbb C}): (R^{\sigma},\eta^{\sigma}) \equiv (R,\eta)\right\}.$$ 

It is a well-known fact that this field is contained in any field of definition of $(R,\eta)$, but it might be that it is not a field of definition. Both the computation of the field of moduli and the determination of whether or not the field of moduli is a field of definition are, in general, difficult problems (see, for instance, \cite{Earle,Hid2,Hid1,yoru2, Huggins,Kontogeorgis, yofm, Shimura}).  A consequence of Weil's descent theorem \cite{Weil} (see also \cite{rubenyo}), the field of moduli is a field of definition if $R$ has no non-trivial automorphisms. A result due to Wolfart \cite{Wolfart} asserts that if $R$ is quasiplatonic (i.e., when $R/{\rm Aut}(R)$ has genus zero and exactly three cone points) then the field of moduli is also a field of definition.

\subsection{The strong field of moduli}
Associated to the pair $(S_{1} \times_{(\beta_{1},\beta_{2})} S_{2},\beta)$, where $\beta=\beta_{1} \circ \pi_{1}=\beta_{2} \circ \pi_{2}$ and $\pi_{j}$ is the corresponding projection, is its field of moduli as in the previous section. 
If $\sigma \in {\rm Gal}({\mathbb C})$, then we have the new pairs $(S_{1}^{\sigma},\beta_{1}^{\sigma})$, $(S_{2}^{\sigma},\beta_{2}^{\sigma})$ and the corresponding pair
$(S_{1}^{\sigma} \times_{(\beta_{1}^{\sigma},\beta_{2}^{\sigma})} S_{2}^{\sigma},\beta^{\sigma})$. 
We say that the pair $(S_{1} \times_{(\beta_{1},\beta_{2})} S_{2},\beta)$ is equivalent to the pair $(S_{1}^{\sigma} \times_{(\beta_{1}^{\sigma},\beta_{2}^{\sigma})} S_{2}^{\sigma},\beta^{\sigma})$, denoted by the symbol
$$(S_{1} \times_{(\beta_{1},\beta_{2})} S_{2},\beta) \equiv^{s} (S_{1}^{\sigma} \times_{(\beta_{1}^{\sigma},\beta_{2}^{\sigma})} S_{2}^{\sigma},\beta^{\sigma}),$$ if there are isomorphisms
$$F_{1}:S_{1} \to S_{1}^{\sigma}, \;\; F_{2}:S_{2} \to S_{2}^{\sigma}$$
so that $\beta^{\sigma}_{j} \circ F_{j}=\beta_{j}$, for $j=1,2.$ ($S_{0}=\widehat{\mathbb C}$; so $S_{0}^{\sigma}=S_{0}$).
The {\it strong field of moduli of the fiber product pair} $(S_{1} \times_{(\beta_{1},\beta_{2})} S_{2},\beta)$ is defined as the 
fixed field of the group $$G=\left\{\sigma \in {\rm Gal}({\mathbb C}): (S_{1} \times_{(\beta_{1},\beta_{2})} S_{2},\beta) \equiv^{s} (S_{1}^{\sigma} \times_{(\beta_{1}^{\sigma},\beta_{2}^{\sigma})} S_{2}^{\sigma},\beta^{\sigma})\right\}.$$

\begin{rema}\label{diferencia}
Let us assume that $(S_{1} \times_{(\beta_{1},\beta_{2})} S_{2},\beta) \equiv^{s} (S_{1}^{\sigma} \times_{(\beta_{1}^{\sigma},\beta_{2}^{\sigma})} S_{2}^{\sigma},\beta^{\sigma}).$ By the definition, there are isomorphisms $F_{j}:S_{j} \to S^{\sigma}_{j}$ so that $\beta^{\sigma}_{j} \circ F_{j}=\beta_{j}$, for $j=1,2.$ We may 
consider the isomorphism
$$F:S_{1} \times_{(\beta_{1},\beta_{2})} S_{2} \to S_{1}^{\sigma} \times_{(\beta_{1}^{\sigma},\beta_{2}^{\sigma})} S_{2}^{\sigma},$$
 by the rule $F(z_{1},z_{2})=(F_{1}(z_{1}),F_{2}(z_{2}))$. This map $F$ satisfies that  $\pi_{j} \circ F=F_{j} \circ \pi_{j}$, for $j=1,2$ (observe that the projection maps $\pi_{j}:S_{1} \times_{(\beta_{1},\beta_{2})} S_{2} \to S_{j}$ are defined over ${\mathbb Q}$, so $\pi_{j}^{\sigma}=\pi_{j}$) and  
 that $\beta^{\sigma} \circ F=\beta$ ($S_{0}^{\sigma}=S_{0}$). So, we may see that the field of moduli of the pair $(S_{1} \times_{(\beta_{1},\beta_{2})} S_{2},\beta)$ is a subfield of its strong field of moduli. 
 \end{rema}

It can be seen from the definitions that the strong field of moduli of the fiber product pair always contains the fields of moduli of the pairs $(S_{1},\beta_{1})$ and $(S_{2},\beta_{2})$. The following result states that, in fact, the smallest field containing these two fields of moduli coincides with the strong field of moduli.

\begin{theo}\label{modulifield}
The strong field of moduli ${\mathcal M}$ of the fiber product pair $(S_{1} \times_{(\beta_{1},\beta_{2})} S_{2},\beta)$ is the smallest field containing the fields of moduli of the pairs $(S_{1},\beta_{1})$ and $(S_{2},\beta_{2})$.
\end{theo}
\begin{proof}
Let us consider the subgroup $G_{j}=\left\{\sigma \in {\rm Gal}({\mathbb C}): (S_{j}^{\sigma},\beta_{j}^{\sigma}) \equiv (S_{j},\beta_{j})\right\}$ and its fixed field ${\mathbb K}_{j}$ (the field of moduli of the pair $(S_{j}, \beta_{j})$), for $j=1,2$. Let ${\mathbb K}$ be the smallest field containing ${\mathbb K}_{1}$ and ${\mathbb K}_{2}$. We need to prove that ${\mathcal M}={\mathbb K}$.
We start by proving that ${\mathcal M} \leq {\mathbb K}.$ Let $\sigma \in {\rm Gal}({\mathbb C})$ acting as the identity on ${\mathbb K}$. As $\mathbb{K}_j \subset \mathbb{K},$ $\sigma$ acts as the identity on ${\mathbb K}_{j}$, so it belongs to $G_{j}$ for $j=1,2$. It follows from the definition that there exists an isomorphism $F_{j}:S_{j} \to S^{\sigma}_{j}$ so that $\beta_{j}=\beta^{\sigma}_{j} \circ F_{j}$. Then $\sigma \in G=\left\{\sigma \in {\rm Gal}({\mathbb C}): (S_{1} \times_{(\beta_{1},\beta_{2})} S_{2},\beta) \equiv^{s} (S_{1}^{\sigma} \times_{(\beta_{1}^{\sigma},\beta_{2}^{\sigma})} S_{2}^{\sigma},\beta^{\sigma})\right\}$, showing that ${\mathcal M} \leq {\mathbb K}$.
Now, we proceed to prove that ${\mathbb K} \leq {\mathcal M}.$ As already observed above, it follows directly from the definitions that the strong field of moduli ${\mathcal M}$ necessarily contains ${\mathbb K}_{1}$ and ${\mathbb K}_{2}$; thus, we have ${\mathbb K}_{1}, {\mathbb K}_{2} \leq {\mathcal M} \leq {\mathbb K}.$ 
As ${\mathbb K}$ is the smallest subfield of ${\mathbb C}$ containing both ${\mathbb K}_{1}$ and ${\mathbb K}_{2}$, we obtain immediately that   
$\mathbb{K} \leq \mathcal{M}$.
\end{proof}

The following is an example for which the strong field of moduli of a fiber product pair contains strictly its field of moduli (as a pair); such an extension has degree two.

\begin{example}\label{Ejemplo:strong} 
Let $(S_{1},\beta_{1})$ be some pair defined over its field of moduli  ${\mathbb Q}(i)$. Example 4.57 in \cite[p. 262]{GiGo} provides such situation in genus zero. An example of genus one is given by taking $S_{1}$ the elliptic curve $y^{2}z=x(x-z)(x-\lambda z)$, where $j(\lambda)=i$ and $j$ is the Klein modular $j$-invariant function (its branch values are $\infty$, $0$ and $1$), and $\beta_{1}(x,y,z)=(j(x/z))^{4}$. Let ${\rm Gal}({\mathbb Q}(i)/{\mathbb Q})=\langle \sigma \rangle$, where $\sigma(i)=-i$. Set $S_{2}=S_{1}^{\sigma}$ and $\beta_{2}=\beta_{1}^{\sigma}$. Consider the pair $(S_2, \beta_2).$ The map
$F:S_{1} \times_{(\beta_{1},\beta_{2})} S_{2} \to S_{1}^{\sigma} \times_{(\beta_{1}^{\sigma},\beta_{2}^{\sigma})} S_{2}^{\sigma}$,  defined by  $F(z_{1},z_{2})=(z_{2},z_{1})$, is an isomorphism between singular Riemann surfaces. As $F$ satisfies
$$\beta^{\sigma} \circ F(z_{1},z_{2})=\beta^{\sigma}(z_{2},z_{1})=(\beta_{1} \circ \pi_{1})^{\sigma}(z_{2},z_{1})=\beta_{1}^{\sigma}(z_{2})=\beta_{2}(z_{2})=\beta(z_{1},z_{2}),$$
the field of moduli of the pair $(S_{1} \times_{(\beta_{1},\beta_{2})} S_{2},\beta)$ is ${\mathbb Q}$.  Since there is no possible isomorphism $F_{1}:S_{1} \to S_{1}^{\sigma}$ satisfying  $\beta_{1}^{\sigma} \circ F=\beta_{1}$ (as the field of moduli of $(S_{1},\beta_{1})$ is different from ${\mathbb Q}$), the strong field of moduli of the fiber product pair is ${\mathbb Q}(i)$.
\end{example}

\subsection{A remark on dessins d'enfants and their fiber products}\label{Sec:dessins}
Belyi's theorem \cite{Belyi} asserts that a compact Riemann surface $S$ can be defined by a curve defined over the field $\overline{\mathbb Q}$ of algebraic numbers if and only if there is a non-constant holomorphic map $\beta:S \to \widehat{\mathbb C}$ whose critical values are contained in the set $\{\infty,0,1\}$; we say that $S$ is a {\it Belyi curve}, that $\beta$ is a {\it Belyi map} for $S$ and that $(S,\beta)$ is a {\it Belyi pair} (or {\it dessin d'enfant}). 
Among all Belyi pairs the most interesting ones are the regular or quasiplatonic ones; these are the ones for which the Belyi map is  a regular branched cover.
Two Belyi pairs, $(S_{1},\beta_{1})$ and $(S_{2},\beta_{2})$, are called equivalent if there exists an isomorphism (holomorphic homeomorphism) $h:S_{1} \to S_{2}$ so that $\beta_{2} \circ h=\beta_{1}$. In this setting, Belyi's theorem asserts that  every Belyi pair $(S,\beta)$ is equivalent to a Belyi pair $(C,\beta_{C})$, where the algebraic curve $C$ and the rational map $\beta_{C}$ are both defined over $\overline{\mathbb Q}$. 
In this way, there is a natural action of the absolute Galois group ${\rm Gal}(\overline{\mathbb Q}/{\mathbb Q})$ on Belyi pairs. Such an action is known to be faithful \cite{GiGo1,GiGo,Gro,Sch}. In his Esquisse d'un Programme \cite{Gro}, Grothendieck pointed out that such an action may provide information on the internal structure of ${\rm Gal}(\overline{\mathbb Q}/{\mathbb Q})$ codified in terms of simple combinatorial objects. 
Let us consider two Belyi pairs $(S_{1},\beta_{1})$ and $(S_{2},\beta_{2})$, where we assume that $S_{j}$ is given as an algebraic curve over $\overline{\mathbb Q}$ and that $\beta_{j}$ is a rational map also defined over $\overline{\mathbb Q}$. Then its fiber product $S_{1} \times_{(\beta_{1},\beta_{2})}S_{2}$ is a connected, possibly reducible, algebraic curve defined over $\overline{\mathbb Q}$ and 
the map $\beta:S_{1} \times_{(\beta_{1},\beta_{2})}S_{2} \to \widehat{\mathbb C}$, defined by $\beta(z_{1},z_{2})=\beta_{j}(z_{j})$, is also a rational map defined over $\overline{\mathbb Q}$. Each irreducible component turns out to be a Belyi curve (and the restriction of $\beta$ to it a Belyi map). 
 If one of the conditions in Theorem \ref{teo1} holds, then the fiber product is irreducible; so it is again a Belyi pair. In this way, the fiber product provides of a tool to construct new dessins d'enfants from two given ones.

\section{Isogenous decomposition of the jacobian variety of fiber products}\label{Sec:jacobiana}
Let $G$ be a finite group acting on a (connected) compact Riemann surface $S$. It is classically known that this action induces an action of $G$ on the Jacobian variety $JS$ of $S$ and this, in turn, gives rise to a $G$-equivariant isogeny decomposition of $JS$ into abelian subvarieties (see, for instance, \cite{cr,l-r}). 
The decomposition of Jacobian varieties with group actions has been extensively studied in different settings, with applications to theta functions, to the theory of integrable systems and to the moduli spaces of principal bundles of curves, among others. The simplest case of such a decomposition is when $G$ is a group of order two; this fact was already noticed in 1895  by Wirtinger \cite{W} and used by Schottky and Jung in \cite{SJ}. For other special groups see, for example, \cite{d1, L2, I, K, d3, rubiyodos, d2, d4}. 
In \cite{KR}, Kani and Rosen studied relations among idempotents in the algebra of rational endomorphisms of an arbitrary abelian variety. By means of these relations, in the case of the Jacobian variety of a  compact Riemann surface $S$ with action of a group $G$, they succeeded in proving a decomposition theorem for $JS$ in which, under some assumptions, each factor is isogenous to the Jacobian of a quotient $S_H$ of $S$ by the action of a subgroup $H$ of $G.$ 
Recently, Rodr\'{\i}guez and the second author in \cite{rubiyo} provided a generalization of Kani-Rosen's result. For the sake of explicitness and for later use, we exhibit a particular case of this generalization.

\begin{prop}[\cite{rubiyo}] \label{rubireyes}
Let $H_1, H_2$ be groups of automorphisms of a (connected) compact Riemann surface $C.$ Then \begin{equation*}
JC \times JC_{\langle H_1, H_2 \rangle} \sim JC_{H_1} \times JC_{H_2} \times P
\end{equation*}
for some abelian subvariety $P$ of $JC$. 
\end{prop}

Let us consider two pairs $(S_{1},\beta_{1})$ and $(S_{2},\beta_{2})$, where $S_{j}$ is a compact Riemann surface and $\beta_{j}:S_{j} \to S_0$ is a regular holomorphic map, over a compact Riemann surface $S_0$, such that the fiber product $S_{1} \times_{(\beta_{1},\beta_{2})} S_{2}$ has structure of a connected Riemann surface. As a consequence of Proposition \ref{rubireyes}, under some conditions which avoid trivial cases, in the next we provide an isogenous decomposition of  the Jacobian $J (S_{1} \times_{(\beta_{1},\beta_{2})} S_{2})$ in such a way that it contains, simultaneously as factors, the Jacobians of the starting Riemann surfaces.

\begin{theo} \label{yo}
Let $(S_{1},\beta_{1})$ and $(S_{2},\beta_{2})$ be two pairs, where $S_{j}$ is a compact Riemann surface and $\beta_{j}:S_{j} \to S_0$ is a regular holomorphic map, over a compact Riemann surface $S_0$. Assume that the set of branch values of $\beta_{1}$ is disjoint from the set of branch values of $\beta_{2}$ and assume that $S_{1}$,  $S_{2}$ and its fiber product are not pairwise isomorphic. Then $$J(S_{1} \times_{(\beta_{1},\beta_{2})} S_{2}) \times JS_0   \sim JS_1 \times JS_2 \times P$$ for a suitable abelian subvariety $P$ of $J(S_{1} \times_{(\beta_{1},\beta_{2})} S_{2})$. In particular, if the genus of $S_0$ is zero, then $$J(S_{1} \times_{(\beta_{1},\beta_{2})} S_{2}) \sim JS_1 \times JS_2 \times P.$$
\end{theo} 

\begin{proof} Let $B_{j} \subset S_{0}$ be the collection of branch values of $\beta_{j}$. We are assuming $B_{1} \cap B_{2} = \emptyset$. We consider a Fuchsian group $\Gamma_{0}$ acting on the hyperbolic plane ${\mathbb H}^{2}$   such that ${\mathbb H}^{2}/\Gamma_{0}$ is the orbifold whose underlying Riemann surface is $S_{0}$ and its cone points  set is $B_{1} \cup B_{2}$ and the cone order are prescribed by $\beta_{1}$ and $\beta_{2}$, respectively. Inside $\Gamma_{0}$ we have, for each $j \in \{1,2\}$, a normal subgroup $\Gamma_{j}$ such that ${\mathbb H}^{2}/\Gamma_{j}$ is an orbifold whose underlying Riemann surface is $S_{j}$ and whose cone points are given by the set $\beta_{j}^{-1}(B_{3-j})$, and the regular cover $\beta_{j}$ is induced by the inclusion $\Gamma_{j}$ in $\Gamma_{0}$.  If we set $G_j \cong \Gamma_{0}/\Gamma_j$, then $G_j \le \mbox{Aut}(S_j)$ is the deck group associated to the holomorphic map $\beta_j$ for $j=1,2$. 
By the results of Section \ref{Sec:Kleinian}, the fiber product $S_{1} \times_{(\beta_{1},\beta_{2})} S_{2}$ (which is irreducible as $B_{1} \cap B_{2}=\emptyset$) is isomorphic to the quotient $\mathbb{H}^2/\Gamma_{12}$ where $\Gamma_{12}=\Gamma_1 \cap \Gamma_2.$ Note that the hypothesis asserts that $\Gamma_{1}$ and $\Gamma_2$ are non-conjugate in $\Gamma_{0}$; in particular, $\Gamma_{1} \neq \Gamma_{2}.$
As mentioned in Remark \ref{obs2-2}, the direct product $G_1 \times G_2$ acts naturally on $S_{1} \times_{(\beta_{1},\beta_{2})} S_{2}$ and is isomorphic to the deck group associated to the holomorphic map $\beta:S_{1} \times_{(\beta_{1},\beta_{2})} S_{2} \to S_{0},$ where $\beta=\beta_{1} \circ \pi_{1}=\beta_{2} \circ \pi_{2}.$ The deck group $H_j \le \mbox{Aut}(S_{1} \times_{(\beta_{1},\beta_{2})} S_{2})$ associated to the projection $\pi_j: S_{1} \times_{(\beta_{1},\beta_{2})} S_{2} \to S_j$ is $$H_1 \cong \Gamma_1/\Gamma_{12} \cong (\Gamma_1\cap \Gamma_0)/(\Gamma_1 \cap \Gamma_2)\cong \{id\} \times G_2 \cong G_2$$ $$H_2 \cong \Gamma_2/\Gamma_{12} \cong (\Gamma_0\cap \Gamma_2)/(\Gamma_1 \cap \Gamma_2)\cong G_1 \times \{id\} \cong G_1.$$
 
Now, as the group generated by $H_1$ and $H_2$ is isomorphic to $G_1 \times G_2,$ the quotient of $S_{1} \times_{(\beta_{1},\beta_{2})} S_{2}$ by the action of $\langle H_1, H_2 \rangle$ is isomorphic to $S_0.$ Thus, Proposition \ref{rubireyes} ensures the existence of an abelian subvariety $P$ of $J(S_{1} \times_{(\beta_{1},\beta_{2})} S_{2})$ such that $$J(S_{1} \times_{(\beta_{1},\beta_{2})} S_{2}) \times JS_0   \sim JS_1 \times JS_2 \times P.$$
\end{proof}



\begin{thebibliography}{99}

\bibitem{Belyi}
G. V. Bely\v{\i}, 
On Galois Extensions of a Maximal Cyclotomic Field. 
{\it Mathematics of the USSR-Izvestiya} {\bf 14} (2) (1980), 247. doi:10.1070/IM1980v014n02ABEH001096

\bibitem{Bers}
L. Bers, 
Spaces of degenerating Riemann surfaces, in: 
{\it Discontinuous groups and Riemann surfaces,} Ann. of  Math. Studies {\bf 79} Princeton Univ. Press (1974), 43--55.

\bibitem{MAGMA}
W. Bosma, J. Cannon and C. Playoust. 
The Magma algebra system. I. The user language.
{\it  J. Symbolic Comput.} {\bf 24} (1997), 235--265.

\bibitem{d1}
A. Carocca, S. Recillas and R. E. Rodr\'iguez. 
Dihedral groups acting on Jacobians.
{\it Contemp. Math.} {\bf 311} (2011), 41--77.

\bibitem{cr} 
A. Carocca and R. E. Rodr\'iguez. 
Jacobians with group actions and rational idempotents.
{\it  J. of Algebra} \textbf{306} (2006), no. 2, 322--343. 

\bibitem{Earle}
C. J. Earle.
On the moduli of closed Riemann surfaces with symmetries.
{\it Advances in the Theory of Riemann Surfaces} (1971), 119--130. Ed. L.V. Ahlfors et al. 
(Princeton Univ. Press, Princeton).


\bibitem{Fried}
M. D. Fried, 
Variables separated equations: strikingly different roles for the branch cycle lemma and the finite simple group classification. 
{\it Sci. China Math.} {\bf 55}  No. 1 (2012), 1--72.

\bibitem{Fulton}
W. Fulton and J. Hansen.
A connectedness theorem for projective varieties, with applications to intersections and singularities of mappings.
{\it Annals of Math.} {\bf 110} (1979), 159-166.



\bibitem{GiGo1}
E. Girondo and G. Gonz\'alez-Diez.
A note on the action of the absolute Galois group on dessins.
{\it Bull. London Math. Soc.} {\bf 39} No. 5 (2007), 721--723.

\bibitem{GiGo}
E. Girondo and G. Gonz\'alez-Diez.
{\it Introduction to compact Riemann surfaces and dessins d'enfants}. 
London Mathematical Society Student Texts {\bf 79}. Cambridge University Press, Cambridge, 2012.


\bibitem{Gro}
A. Grothendieck.
Esquisse d'un Programme (1984). In Geometric Galois Actions. 
L. Schneps and P. Lochak eds., {\it London Math. Soc. Lect. Notes Ser.} {\bf 242}.
Cambridge University Press, Cambridge, 1997, 5--47.


\bibitem{Hid2}
R. A. Hidalgo.
Erratum to: Non-hyperelliptic Riemann surfaces with real field of moduli but not definable over the reals.
{\it Archiv der Math.} {\bf 98} (2012), 449--451.


\bibitem{H:fiberproduct}
R. A. Hidalgo.
The fiber product of Riemann surfaces: A Kleinian group point of view.
{\it Analysis and Mathematical Physics} {\bf 1} (2011), 37-45.

\bibitem{Hid1}
R. A. Hidalgo.
Non-hyperelliptic Riemann surfaces with real field of moduli but not definable over the reals.
{\it Archiv der Mathematik} {\bf 93} (2009), 219--222. 


\bibitem{L2}
R. A. Hidalgo, L. Jim\'enez, S. Quispe and  S. Reyes-Carocca, Quasiplatonic curves with symmetry group $\mathbb{Z}_2^2 \rtimes \mathbb{Z}_m$ are definable over $\mathbb{Q}.$ {\it Bull. London Math. Soc}. {\bf 49} (2017) 165--183.


\bibitem{rubenyo}
R. A. Hidalgo and S. Reyes-Carocca. 
A constructive proof of Weil's Galois descent theorem.
arXiv:1203.6294.

\bibitem{yoru2}
R. A. Hidalgo and S. Reyes-Carocca, 
Field of moduli of classical Humbert curves. 
{\it  Q. J. Math}. {\bf 63} (2012), no. 4, 919--930.

\bibitem{Huggins}
B. Huggins.
Fields of moduli of hyperelliptic curves.
{\it Math. Res. Lett.} {\bf 14} No. 2 (2007), 249--262.

\bibitem{Iitaka}
S. Iitaka.
{\it Algebraic geometry: An introduction to the birational geometry of algebraic varieties}. 
Springer Verlag, 1982.

\bibitem{I}
M. Izquierdo, L. Jim\'enez and A. M. Rojas. 
Decomposition of Jacobian varieties of curves with dihedral actions via equisymmetric stratification. To appear in Rev Mat. Iberoamericana. arXiv:1609.01562.


\bibitem{K} 
V. Kanev. 
Spectral curves, simple Lie algebras, and Prym-Tjurin varieties. 
Theta functions--Bowdoin 1987, Part 1, 627--645, {\it Proc. Sympos. Pure Math.} {\bf 49}, Part 1, Amer. Math. Soc., Providence, RI, 1989. 


\bibitem{KR} 
R. Kani and M. Rosen.  
Idempotent relations and factors of Jacobians. 
{\it Math. Ann.} {\bf 284} (1989) 307--327.

\bibitem{Kontogeorgis}
A. Kontogeorgis. 
Field of moduli versus field of definition for cyclic covers of the projective line.
{\it J. de Theorie des Nombres de Bordeaux} {\bf 21} (2009) 679--692.


\bibitem{l-r}
H. Lange and S. Recillas. 
Abelian varieties with group actions.
{\it J. Reine Angew. Math.} \textbf{575} (2004) 135--155.

%

\bibitem{Massey}
W. Massey. 
{\it A Basic Course in Algebraic Topology.}
 New York: Springer (1991). ISBN 0-387-97430-X. 

\bibitem{Pakovich}
F. Pakovich. 
Prime and composite Laurent polynomials. 
{\it Bull. Sci. Math.} {\bf 133} No. 7 (2009), 693--732.


\bibitem{d3}
S. Recillas and R. E. Rodr\'iguez. 
Jacobians and representations of $S_3$. 
Aportaciones Mat. Investig. {\bf 13}, {\it Soc. Mat. Mexicana}, M\'exico (1998).

\bibitem{yofm}
S. Reyes-Carocca, 
Field of moduli of generalized Fermat curves. 
{\it  Q. J. Math}. {\bf 63} (2012), no. 2, 467--475. 

\bibitem{rubiyo}
S. Reyes-Carocca and R. E. Rodr\'iguez. 
A generalisation of Kani-Rosen decomposition theorem for Jacobian varieties.
To appear in {\it Annali della Scuola Normale Superiore di Pisa, Classe di Scienze}
DOI: 10.2422/2036-2145.201706$\textunderscore$003

\bibitem{rubiyodos}
S. Reyes-Carocca and R. E. Rodr\'iguez. 
On Jacobians with group action and converings, Preprint, Arxiv: http://arxiv.org/abs/1711.07552.

\bibitem{d2}
J. Ries. 
The Prym variety for a cyclic unramified cover of a hyperelliptic curves.
{\it J. Reine Angew. Math.} \textbf{340} (1983) 59--69.

\bibitem{Roch}
G. Roch.
Ueber die Anzahl der willkurlichen Constanten in algebraischen Functionen.
{\it Journal f\"ur die reine und angewandte Mathematik} {\bf  64} (1865), 372-376.

\bibitem{d4}
A. S\'anchez-Arg\'aez. 
Actions of the group $A_5$ in Jacobian varieties. 
{\it Aportaciones Mat. Comun}. {\bf 25}, Soc. Mat. Mexicana, México (1999), 99--108.


\bibitem{Sch}
L. Schneps.
Dessins d'enfant on the Riemann sphere.
In {\it The Grothendieck theory of dessins d'enfants}. Edited by Leila Schneps. 
{\it London Math. Soc. Lect. Notes Ser.} {\bf 200}. 
Cambridge University Press, Cambridge, 1994, 47--77.

\bibitem{SJ}
F. Schottky and H. Jung. 
Neue S\"atze \"uber Symmetralfunctionen und due Abel'schen Functionen der Riemann'schen Theorie. 
S.B. Akad. Wiss. (Berlin) {\it Phys. Math. Kl.} {\bf 1} (1909), 282--297.

\bibitem{Shimura}
G. Shimura.
On the field of rationality for an abelian variety.
{\it Nagoya Math. J.} {\bf 45} (1972), 167--178.


\bibitem{W}
W. Wirtinger.  
Untersuchungen \"uber Theta Funktionen, Teubner, Berlin, 1895.

\bibitem{Weil}
A. Weil.
The field of definition of a variety.
{\it  Amer. J. Math.} {\bf 78} (1956), 509--524.

 \bibitem{Wolfart}
J.  Wolfart.
 $ABC$ for polynomials, dessins d'enfants and uniformization---a survey. Elementare und analytische Zahlentheorie, 313--345, Schr. Wiss. Ges. Johann Wolfgang Goethe Univ. Frankfurt am Main, 20, Franz Steiner Verlag Stuttgart, Stuttgart, 2006.





\end{thebibliography}
\end{document}